\documentclass[11pt]{amsart}
\usepackage{amssymb, latexsym}
\theoremstyle{plain}
\newtheorem{theorem}{Theorem}

\newtheorem {lemma}{Lemma}
\newtheorem{proposition}{Proposition}

\theoremstyle{remark}

\newtheorem*{Remark 1}{Remark 1}
\newtheorem*{Remark 2}{Remark 2}
\newtheorem*{Remark 3}{Remark 3}
\newtheorem*{Remark 4}{Remark 4}

\numberwithin{equation}{section}

\begin{document}

\title[Asymptotics for Diffusion Processes  with Random Jumps]%
 {Asymptotics for  Exit Problem and Principal Eigenvalue  for a
 Class of Non-Local Elliptic Operators   Related to Diffusion Processes with Random Jumps and Vanishing Diffusion}

\author{ Ross G. Pinsky}
\address{Department of Mathematics\\
Technion---Israel Institute of Technology\\
Haifa, 32000\\ Israel}
\urladdr{http://www.math.technion.ac.il/~pinsky/}
\thanks{This  research   was supported by
 THE ISRAEL SCIENCE FOUNDATION (grant No.
449/07)}

\date{}

\begin{abstract}
Let $D\subset R^d$ be a bounded domain
and denote by $\mathcal P(D)$  the space of probability measures on $D$.
Let
\begin{equation*}
L=\frac12\nabla\cdot a\nabla +b\nabla
\end{equation*}
be a second order elliptic operator.
Let $\mu\in\mathcal P(D)$ and   $\delta>0$.
 Consider a Markov process $X(t)$ in $D$ which performs diffusion in $D$ generated by the operator $\delta L$  and is stopped at the
boundary, and which while running, jumps instantaneously,  according to an exponential clock
with spatially dependent intensity $V>0$, to a new point,  according to the distribution $\mu$.
The Markov process is generated by the operator $L_{\delta,\mu, V}$
defined by
\begin{equation*}
L_{\delta,\mu, V}\phi\equiv \delta L \phi+V(\int_D\phi~d\mu-\phi).
\end{equation*}
Let $\phi_{\delta,\mu,V}$ denote the solution to the Dirichlet problem
\begin{equation*}\label{Dirprob}
\begin{aligned}
&L_{\delta,\mu,V}\phi=0\ \text{in}\ D;\\
&\phi=f\ \text{on}\ \partial D,
\end{aligned}
\end{equation*}
where $f$ is  continuous.
The solution has
the stochastic
representation
\begin{equation*}
\phi_{\delta,\mu,V}(x)=E_xf(X(\tau_D)).
\end{equation*}
One has that $\phi_{0,\mu,V}(f)\equiv\lim_{\delta\to0}\phi_{\delta,\mu,V}(x)$ is independent of $x\in D$.
We evaluate this constant in the case that $\mu$ has a density in a neighborhood
of $\partial D$.
We also study the asymptotic behavior as $\delta\to0$
of the principal eigenvalue $\lambda_0(\delta,\mu,V)$ for the operator
$L_{\delta,\mu, V}$, which generalizes previously obtained results for the case $L=\frac12 \Delta$.
\end{abstract}

\maketitle
\section{Introduction and Statement of Results}

Let $D\subset R^d$ be a bounded domain with $C^{2,\alpha}$-boundary ($\alpha\in (0,1]$) and let $\mathcal P(D)$ denote the space of probability measures on $D$.
Let
\begin{equation}
L=\frac12\nabla\cdot a\nabla +b\nabla
\end{equation}
be a second order elliptic operator. Assume that the coefficients   $a=\{a_{i,j}\}_{i,j=1}^n$
and $b=\{b_i\}_{i=1}^n$ are  in $C^{1,\alpha}(\bar D)$ and that $a(x)$ is positive definite for each $x\in\bar D$.
Fix a measure $\mu\in\mathcal P(D)$ and fix  $\delta>0$.
 Consider a Markov process $X(t)$ in $D$ which performs diffusion in $D$ generated by the operator $\delta L$  and is stopped at the
boundary, and which while running, jumps instantaneously,  according to an   exponential clock
with spatially dependent intensity $V$, to a new point,  according to the distribution $\mu$.
That is, the probability that the process $X(\cdot)$ has not jumped by time $t$ is given by $\exp(-\int_0^{t\wedge\tau_D} V(X(s))ds)$,
where $\tau_D=\inf\{t\ge 0: X(t)\not\in D\}$ is the first exit time from $D$.
 From its new position after the jump,
the process repeats the above behavior independently of what has
transpired previously. We assume that $V$ is positive and in $C^\alpha(\bar D)$.
 Denote
probabilities and expectations for the process
starting from $x\in D$ by $P_x^{\delta,\mu,V}$ and
$E_x^{\delta,\mu,V}$.

Let $L_{\delta,\mu,V}$ denote the operator defined by
\begin{equation*}
L_{\delta,\mu, V}\phi\equiv \delta L \phi+V(\int_D\phi~d\mu-\phi).
\end{equation*}
The operator $L_{\delta,\mu, V}$ generates the Markov process $X(t)$, and consequently,
$\phi(X(t\wedge\tau_D))-\int_0^{t\wedge\tau_D}L_{\delta,\mu, V}\phi(X(s))ds$
is a martingale.

Let $\phi_{\delta,\mu,V}$ denote the solution to the Dirichlet problem
\begin{equation}\label{Dirprob}
\begin{aligned}
&L_{\delta,\mu,V}\phi=0\ \text{in}\ D;\\
&\phi=f\ \text{on}\ \partial D,
\end{aligned}
\end{equation}
where $f$ is  continuous.
It follows that $\phi_{\delta,\mu,V}(X(t\wedge \tau_D))$ is a martingale; thus $\phi(x)=E_x\phi_{\delta,\mu,V}(X(t\wedge\tau_D))$, for all t.
Letting $t\to\infty$ gives the stochastic
representation
\begin{equation}\label{stochrep}
\phi_{\delta,\mu,V}(x)=E_xf(X(\tau_D)).
\end{equation}
In this paper we investigate the behavior of $\phi_{\delta,\mu,V}$ as $\delta\to0$; that is, in the small diffusion limit.
Since $L_{0,\mu,V}\phi=V(x)(\int_D\phi d\mu-\phi(x))$, one expects that $\lim_{\delta\to0}\phi_{\delta,\mu,V}(x)$ will be
independent of $x\in D$, and we can prove this trivially  via the stochastic representation in \eqref{stochrep}.
We wish to calculate the constant
\begin{equation}\label{const}
 \phi_{0,\mu,V}(f)\equiv\lim_{\delta\to0}\phi_{\delta,\mu,V}(x), \ x\in D.
\end{equation}

The above constant depends very strongly on the behavior of $\mu$ near the boundary. Here we treat the
case that supp$(\mu)\cap\partial D\neq\emptyset$ and that $\mu$ has a density in a neighborhood of the boundary. 
The density may vanish on the boundary.
Let
$\tilde L$ denote the formal adjoint of $L$:
$$
\tilde L=\frac12\nabla\cdot a\nabla -b\nabla-\nabla\cdot b.
$$

\begin{theorem}\label{th}
Let $D\subset R^d$, $d\ge1$, be a bounded domain with a $C^{2,\alpha}$-boundary ($\alpha\in(0,1]$) and
let $\mu\in\mathcal P(D)$. Assume that $V>0$ on $\bar D$.
Let $D^\epsilon=\{x\in D: \text{dist}(x,\partial D)<\epsilon\}$.

\noindent  Assume that
for some
$\epsilon>0$, the restriction of $\mu$ to $D^\epsilon$ possesses a
density: $\mu(dx)|_{D^\epsilon}\equiv\mu(x)dx$.
Assume that for some $k\ge0$, the following conditions hold.
If $k$ is even, assume that $\mu\in C^k(\bar D^\epsilon)$;  if $k$ is odd,
assume that $\mu\in C^{k+1}(\bar D^\epsilon)$.
 Assume that
$$
\begin{aligned}
&\frac{d^\beta \mu}{d x^\beta}\equiv0 \ \text{on}\ \partial D,\ \text{ for all}\ |\beta|\le k-1,\ \text{if}\ k\ge1; \\
&\frac{d^\beta \mu}{d x^\beta}\not\equiv0 \ \text{on}\ \partial D,\ \text{for some}\ |\beta|=k,\ \text{if}\ k\ge0.
\end{aligned}
$$
Assume that $V\in C^{2,\alpha}(\bar D)$, if $k=0,2$, and that $V\in C^k(\bar D)$, if $k\ge4$ is even;
assume   that $V\in C^{k+1}(\bar D)$, if $k$ is odd.

\noindent Assume that  $a_{i,j}, b_i\in C^{1,\alpha}(\bar D)$, if $k=0,2$, and that $a_{i,j}, b_i\in C^{k-1}(\bar D)$,
if $k\ge4$ is even; assume that $a_{i,j}, b_i\in C^{1,\alpha}(\bar D)$, if $k=1$, and that $a_{i,j}, b_i\in C^k(\bar D)$, if
$k\ge3$ is odd.

\noindent Let $n$ denote the inward unit normal to $D$ at $\partial D$. Let $\sigma$ denote Lebesgue measure on $\partial D$.
\newline
If $k$ is even, then the solution $\phi_{\delta,\mu,V}$ to \eqref{Dirprob} satisfies \eqref{const} with
\begin{equation}\label{even}
\phi_{0,\mu,V}(f)=\frac{\int_{\partial D}
f\sqrt{(n\cdot an)}V^{-\frac{k+1}2}\tilde L^{\frac k2}\mu d\sigma}{\int_{\partial D}\sqrt{(n\cdot an)}V^{-\frac{k+1}2}\tilde L^{\frac k2}\mu d\sigma}.
\end{equation}
If $k$ is odd, then  the solution $\phi_{\delta,\mu,V}$ to \eqref{Dirprob} satisfies \eqref{const} with
\begin{equation}\label{odd}
\phi_{0,\mu,V}(f)=\frac{\int_{\partial D}f
V^{-\frac{k+1}2}a\nabla(\tilde L^{\frac{k-1}2}\mu)\cdot nd\sigma}{\int_{\partial D}V^{-\frac{k+1}2}a\nabla(\tilde L^{\frac{k-1}2}\mu)\cdot nd\sigma}.
\end{equation}
In particular then, if $k=0$, one has
\begin{equation}
\phi_{0,\mu,V}(f)=\frac{\int_{\partial D}f\sqrt{(n\cdot an)}\frac\mu{\sqrt {V}}d\sigma}{\int_{\partial D}\sqrt{(n\cdot an)}\frac\mu{\sqrt {V}}d\sigma},
\end{equation}
and if $k=1$, one has
\begin{equation}
\phi_{0,\mu,V}(f)=\frac{\int_{\partial D}f
V^{-1}a\nabla \mu\cdot nd\sigma}{\int_{\partial D}V^{-1}a\nabla\mu\cdot nd\sigma}.
\end{equation}
\end{theorem}

\medskip

\noindent \bf Remark 1.\rm\ Note that if $k=0$ or $k=1$, then
$\phi_{0,\mu,V}(f)$ depends  on the diffusion coefficient $a$, but not on the drift coefficient $b$, whereas for $k\ge2$,
$\phi_{0,\mu,V}(f)$ depends on  $a$ and $b$. 
\medskip

\noindent\bf Remark 2.\rm\ If $\mu$ has compact support,
the behavior of $\phi_{0,\mu,V}(f)$ is completely different.
In this case, $\phi_{0,\mu,V}(f)$ can be studied using
Wentzell-Freidlin action functionals. Assuming the uniqueness of
the minimum of a certain such functional, $\phi_{0,\mu,V}(f)$ will have the form $f(x_0)$ for some
$x_0\in\partial D$.
\medskip

Define the contraction semigroup
\begin{equation*}
T_t^{\delta,\mu,V}f(x)=E_x^{\delta,\mu,V}(f(X(t)); \tau_D>t),\  f\in C_0(\bar D),
\end{equation*}
where $C_0(\bar D)$ is the space of continuous functions on $\bar D$ vanishing on $\partial D$.
The infinitesimal generator of this semigroup is an extension of
the operator $L_{\delta,\mu,V}$, defined on
$C^2(\bar D)\cap\{\phi:\phi, L_{\delta,\mu,V}\phi\in C_0(\bar D)\}$
with the homogeneous Dirichlet boundary condition.
The  operator $T_t^{\delta,\mu,V}$ is  compact (see \cite{P09} where the case of
constant coefficients is considered); thus, the resolvent operator
for $T_t^{\delta,\mu,V}$ is also  compact, and consequently  the spectrum $\sigma(L_{\delta,\mu,V})$ of $L_{\delta,\mu,V}$
consists exclusively of eigenvalues.
By the Krein-Rutman theorem, one deduces that $-L_{\delta,\mu,V}$ possesses a principal eigenvalue,
$\lambda_0(\delta,\mu,V)$; that is, $\lambda_0(\delta,\mu,V)$ is real and simple  and satisfies
$\lambda_0(\delta,\mu,V)=\inf\{\text{Re}(\lambda):\lambda\in\sigma(-L_{\delta,\mu,V})\}$ \cite{P95}.
It is known that $\lambda\in\sigma(-L_{\delta,\mu,V})$
if and only if $\exp(-\lambda t)\in\sigma(T_t^{\delta,\mu,V})$ \cite{Pazy83}.
Thus, since $||T^{\delta,\mu,V}_t||<1$, it follows that $\lambda_0(\delta,\mu,V)>0$.
We have
$$
\sup_{f\in C_0(\bar D),||f||\le1}||T_t^{\delta,\mu,V}f||=\sup_{x\in D}P_x^{\delta,\mu,V}(\tau_D>t);
$$
 thus, a standard result \cite{RS72} allows us to conclude that
$$
\lim_{t\to\infty}\frac1t\log \sup_{x\in D}P_x^{\delta,\mu,V}(\tau_D>t)=-\lambda_0(\delta,\mu,V).
$$
It is well known that this is equivalent to
\begin{equation}\label{largetime}
\lim_{t\to\infty}\frac1t\log P_x^{\delta,\mu,V}(\tau_D>t)=-\lambda_0(\delta,\mu,V),\ \ x\in D.
\end{equation}
In the case that $L=\frac12\Delta$, that is the case that the underlying motion is Brownian motion,
the papers   \cite{P09}, \cite{AP11} investigated  the behavior of the principal eigenvalue $\lambda_0(\delta,\mu,V)$ as $\delta\to0$.
(Actually, in those papers, one finds the operator
$\gamma L_{\frac1\gamma, \mu,V}\phi=\frac12 \Delta\phi +\gamma V(\int_D\phi d\mu-\phi)$ with $\gamma\to\infty$.)
The key calculations contained in Proposition \ref{key} of this paper for the case of a general diffusion operator $L$  generalize
calculations in \cite{AP11} for the operator $\frac12\Delta$.
Using the methods of \cite{AP11} along with Proposition \ref{key}, one obtains the following generalization of the results
in \cite{AP11}.
\begin{theorem}\label{th-ev}
Let the assumptions  of Theorem \ref{th} be in effect for some $k\ge0$.
Then
the principal eigenvalue $\lambda_0(\delta,\mu,V)$ of the operator $-L_{\delta,\mu,V}$
behaves asymptotically as follows:

\noindent i. If $k$ is even,
\begin{equation}\label{odd}
\lim_{\delta\to0}\delta^{-\frac{k+1}2}\lambda_0(\delta,\mu,V)=\frac{\int_{\partial D}\sqrt{(n\cdot an)}V^{-\frac{k+1}2}\tilde L^{\frac k2}\mu d\sigma}
{\sqrt2\int_D\frac1{ V}d\mu};
\end{equation}

\noindent ii. If $k$ is odd,
\begin{equation}\label{even}
\lim_{\delta\to0}\delta^{-\frac{k+1}2}\lambda_0(\delta,\mu,V)=\frac{\int_{\partial D}V^{-\frac{k+1}2}a\nabla(\tilde L^{\frac{k-1}2}\mu)\cdot nd\sigma}
{2\int_D\frac1{ V}d\mu}.
\end{equation}

\end{theorem}
\medskip

\noindent \bf Remark 1.\rm\
 Note that if $k=0$ or $k=1$, then the leading asymptotic behavior of
$\lambda_0(\delta,\mu,V)$ depends  on the diffusion coefficient $a$, but not on the drift coefficient $b$, whereas for $k\ge2$,
it depends on $a$ and $b$.
\medskip

\bf \noindent Remark 2.\rm\ We note that if $\mu$ has compact support, then there exist constants
$c_1,c_2>0$ such that $\exp(-c_2\delta^{-\frac12})\le \lambda_0(\delta,\mu,V)\le \exp(-c_1\delta^{-\frac12})$,
for small $\delta>0$. This was proven in  \cite{P09, AP11} for the case $L=\frac12\Delta$. The same type of proof works for general $L$.
\medskip

If $\mu\not\equiv0$ on $\partial D$, then Theorem \ref{th-ev} gives
\begin{equation}\label{basiccase}
\lambda_0(\delta,\mu,V)\sim\frac{\int_{\partial D}\sqrt{(n\cdot an)}\frac{\mu}{V^\frac12}d\sigma}{\sqrt2\int_D\frac1Vdx}\delta^\frac12,\
\text{as}\ \delta\to0.
\end{equation}
Theorem \ref{th-ev} is proven under the assumption that $V>0$ in $\bar D$. This condition is essential.
Note that if $V$ vanishes in a  sub-domain $A\subset D$, then
as long as the process $X(t)$ remains in $A$, it never jumps, and thus starting from a point in $A$, 
the probability that $X(t)$ does not exit  $D$ by time $t$ is greater than the probability that a $\delta L$-diffusion
process does not exit $A$ by time $t$; thus, in light \eqref{largetime} and the corresponding equation  for
the $\delta L$ diffusion process, it follows that  $\lambda_0(\delta,\mu,V)\le \delta \lambda_0^A(L)$, where
$\lambda_0^A(L)>0$ is the principal eigenvalue for $-L$ in $A$.
In particular, $\lambda_0(\delta,\mu,V)$ is on 
a smaller order than in \eqref{basiccase}.

Now consider the case that $V>0$ in $D$ but $V\equiv0$ on $\partial D$.
On the one hand, since the process  needs to  \bf not\rm\ jump in order to exit $D$,  allowing the jump mechanism to
weaken at the boundary should help the process exit. Thus, if $\mu\not\equiv0$ on $\partial D$, one might expect
that $\lambda_0(\delta,\mu,V)$ will be on a larger order than $\delta^\frac12$.
But on the other hand, if $V\equiv0$ in $D^\epsilon$, then by the argument in the previous paragraph,
$\lambda_0(\delta,\mu,V)\le\lambda^{D^\epsilon}_0(L)\delta$.
Thus, we expect that if $V$ vanishes identically on the boundary to high enough order, then
$\lambda_0(\delta,\mu,V)$ will be on a smaller order than $\delta^\frac12$.

Now let $V_\epsilon\equiv\epsilon+\hat V$,where $\hat V>0$ in $D$ and $\hat V\equiv0$ on $\partial D$,
and  substitute  $V_\epsilon$ for $V$ in the righthand side of \eqref{basiccase}.
If $\hat V$  vanishes to the  first order on $\partial D$, then
the right hand side of \eqref{basiccase} is on the order  $(\epsilon^\frac12\log\epsilon)^{-\frac12}$;
in particular, it converges to $\infty$.  This suggests that if $V$ vanishes to  first
order on $\partial D$, then $\lambda_0(\delta,\mu,V)$ will  be on a larger order than $\delta^\frac12$.
If $\hat V$  vanishes to  second order on $\partial D$, then
the right hand side of \eqref{basiccase} stays bounded and bounded from 0
as $\epsilon\to0$. This suggests that if $V$ vanishes to second order, then
$\lambda_0(\delta,\mu,V)$  will be on the order $\delta^\frac12$, as in \eqref{basiccase}.
If $\hat V$  vanishes to  third order on $\partial D$, then
the right hand side of \eqref{basiccase} goes to  0
as $\epsilon\to0$.  This suggests that if $V$ vanishes to third order or higher, then
$\lambda_0(\delta,\mu,V)$ will be on a smaller order than $\delta^\frac12$.

\bf \noindent Open Question:\it\ Consider the case that $\mu\not\equiv0$ on $\partial D$, so that if $V$ were strictly positive in
$\bar D$, then \eqref{basiccase} would hold.  Assume that  $V>0$ in $D$ and that $V$ vanishes identically
on $\partial D$ to the order $k$, $k\ge1$.  At what order does 
$\lambda_0(\delta,\mu,V)$ approach 0 when $\delta\to0$?\rm\

\medskip

In section 2 we present several auxiliary results which then allow for a quick  proof of Theorem \ref{th}. The proof
of one of the auxiliary results is deferred  to section 3.

\section{Auxiliary Results and Proof of Theorem \ref{th}}
In this section we present three lemmas and one proposition, from which the theorem will follow quickly.
We begin however with a useful construction of the
 process $X(\cdot)$ up to its exit time from $\partial D$.
On a common probability space with probability measure $\mathcal{P}^\delta$, let $\{Y_n\}_{n=0}^\infty$ be an independent sequence of diffusion
processes, where each $Y_n(\cdot)$ is a diffusion corresponding to the operator $\delta L$ and stopped upon reaching the
boundary, and where   $Y_0(0)=x\in D$,  and the  distribution of $Y_n(0)$
for $n\ge1$ is $\mu$. (We have purposely suppressed the dependence of $\mathcal{P}^\delta$ on $x$. Note that
$Y_n$ does not depend on $x$ for $n\ge1$.)
Let $\mathcal{F}_{t,n}=\sigma(Y_n(s),0\le s\le t)$ denote the filtration up to time $t$ for the process $Y_n(\cdot)$.
Denote the exit time of $Y_n(\cdot)$ from $D$ by $\tau_{D,n}$.
Let $J_n$ be a stopping time for $Y_n(\cdot)$ satisfying
$\mathcal{P}^\delta(J_n>t|\mathcal{F}_{t,n})=\exp(-\int_0^{t\wedge\tau_{D,n}}V(Y_n(s))ds)$.
Now define by induction:
$$
\begin{aligned}
&X(t)=Y_0(t),\ \text{for}\ 0\le t<J_0\wedge\tau_{D,0};\\
& \text{if}\ J_{n-1}<\tau_{D,n-1},\ \text{then}\\
&X(t)=Y_n(t-\sum_{k=0}^{n-1}J_k),\ \text{for}\ \sum_{k=0}^{n-1}J_k\le t<\sum_{k=0}^{n-1}J_k+J_n\wedge \tau_{D,n}.
\end{aligned}
$$

Let $e_{\delta,x}(\cdot)=P_x^{\delta,\mu,V}(X(\tau_D)\in \cdot)$ denote the exit distribution of the process $X(\cdot)$ from $D$
starting from $X(0)=x$.
Let $J$ denote the first jump time of $X(\cdot)$. Since the distribution of
$\tau_{D,0}$ converges to the point mass at $\infty$ as $\delta\to0$,
we have $\lim_{\delta\to0}\mathcal{P}^\delta(\tau_{D,0}<J_0)=0$; equivalently,
 $\lim_{\delta\to0}P_x^{\delta,\mu,V}(\tau_D<J)=0$.
Since $P_x^{\delta,\mu,V}(X(\tau_D)\in\cdot|J<\tau_D)$ is independent of $x\in D$, we conclude
that the weak limit of  $e_{\delta,x}$ (which will be shown to exist) is independent of $x\in D$.
Indeed, from the above considerations  and the above  construction of $X(\cdot)$, we have
$$
\lim_{\delta\to0}e_{\delta,x}(\cdot)=\lim_{\delta\to0}\mathcal{P}^\delta(Y_1(\tau_{D,1})\in\cdot|\tau_{D,1}< J_1).
$$
(Note that the 1 appearing in four places on the righthand side above can be replaced by any $n\ge2$ without changing the
value of the expression.)

For each $k=0,1,\cdots$, let  $e_0^k(\cdot)$ denote the probability measure on $\partial D$ with density
$e_0^k(x)$ given by
$$
\begin{aligned}
&e_0^k(x)=\Large(\int_{\partial D}\sqrt{(n\cdot an)}V^{-\frac{k+1}2}\tilde L^{\frac k2}\mu d\sigma\Large)^{-1}
\sqrt{(n\cdot an)(x)}V^{-\frac{k+1}2}(x)\tilde L^{\frac k2}\mu(x),\\
& k\ge 0,  k \ \text{even};\\
&e_0^k(x)=\Large(\int_{\partial D}V^{-\frac{k+1}2}a\nabla(\tilde L^{\frac{k-1}2}\mu)\cdot nd\sigma\Large)^{-1}~ V^{-\frac{k+1}2}(x)a(x)(\nabla(
\tilde L^{\frac{k-1}2}\mu)\cdot n)(x),\\
& k\ge1, k\ \text{odd}.
\end{aligned}
$$
From \eqref{stochrep}, to prove the theorem we need to prove that
if $\mu$ satisfies the conditions  of the theorem for a particular $k\ge0$, then
$e_{\delta,x}$ converges weakly
to $e_0^k$.

From now on we assume that $\mu$ satisfies the conditions of the theorem for a particular $k\ge0$.
Fix $m\ge1$ and let $\{A_j\}_{j=1}^{m+1}$ be a partition of $\bar D$ into $m+1$ disjoint connected  sets satisfying the following conditions:
(i) $A_j$ has a nonempty interior, for all $j$; (ii) $A_j\cap \partial D$ has a nonempty interior in the relative topology of $\partial D$, for
all $j\neq m+1$; (iii)  $e_0^k(A_1\cap\partial D)>0$;
(iv) $\text{dist}(A_{m+1},\partial D)>0$.
To prove the theorem, it is enough to show that
\begin{equation}\label{toshow}
\lim_{\delta\to0}e_\delta^k(A_1\cap\partial D)=e_0^k(A_1\cap \partial D).
\end{equation}

Let $u_{\delta,V}$ denote the solution
to
\begin{equation}\label{u}
\begin{aligned}
&\delta Lu-Vu=0 \ \text{in}\ D;\\
&u=1\ \text{on}\ \partial D.
\end{aligned}
\end{equation}
Let $\mathcal{E}^\delta$ denote expectations corresponding to $\mathcal{P}^\delta$.
As is well-known, $u_{\delta, V}$ has the stochastic representation
$$
u_{\delta,V}(x)=\mathcal{E}^\delta(\exp(-\int_0^{\tau_{D,1}}V(Y_1(t)))|Y_1(0)=x).
$$
\begin{lemma}\label{intu}
\begin{equation}
\mathcal{P}^\delta(Y_1(0)\in A_j,\tau_{D,1}<J_1)=\int_{A_j}u_{\delta,V}~d\mu, \ j=1,\cdots, m+1.
\end{equation}
\end{lemma}
\begin{proof}
For $x\in D$, we have
$$
\begin{aligned}
&\mathcal{P}^\delta(\tau_{D,1}<J_1|Y_1(0)=x)=
\int_0^\infty \mathcal{P}^\delta(J_1>t,\tau_{D,1}=dt|Y_1(0)=x)=\\
&\int_0^\infty \mathcal{E}^\delta\mathcal{P}^\delta(J_1>t,\tau_{D,1}=dt|\mathcal{F}_{t,1}, Y_1(0)=x)=\\
&\int_0^\infty \mathcal{E}^\delta(\exp(-\int_0^{t\wedge\tau_{D,1}}V(Y_1(s))ds)1_{dt}(\tau_{D,1})|Y_1(0)=x)=\\
&\mathcal{E}(-\int_0^{\tau_{D,1}}V(Y_1(t))dt|Y_1(0)=x)=u_{\delta,V}(x).
\end{aligned}
$$
The result follows from this.
\end{proof}
\begin{lemma}\label{expdecay}
There exists a constant $c=c(a,b,d, V)>0$ depending on the coefficients $a,b$ of $L$, on $V$ and on the dimension $d$ such that
$$
u_{\delta, V}(x)\le c\exp(-\frac{\text{dist}(x,\partial D)}{c\delta^{\frac12}}).
$$
In particular then by Lemma \ref{intu}, for some $c_1>0$,
\begin{equation}\label{expdecayprob}
\mathcal{P}^\delta(Y_1(0)\in A_{m+1},\tau_{D,1}<J_1)\le c_1\exp(-\frac{\text{dist}(A_{m+1},\partial D)}{c_1\delta^{\frac12}}).
\end{equation}
\end{lemma}
\begin{proof}
Let $V_\text{min}=\min_{x\in \bar D}V(x)$.
From the stochastic representation of $u_{\delta, V}$, we have for any $t>0$,
\begin{equation}\label{bound}
\begin{aligned}
&u_{\delta, V}(x)\le \exp(-tV_{\text{min}})+\mathcal{P}^\delta(\tau_{D,1}\le t|Y_1(0)=x)\le\\
&\exp(-tV_{\text{min}})+\mathcal{P}^\delta(\max_{0\le s\le t}|Y_1(s)-x|\ge \text{dist}(x,\partial D)|Y_1(0)=x).
\end{aligned}
\end{equation}
Since $L$ has been written in divergence form, the drift of $Y_1$ is  $\delta(b+\frac12\nabla\cdot a)$.
Let $B=\max_{x\in \bar D}|b(x)+\frac12\nabla\cdot a(x) |$ and let $A=\max_{|v|=1}\max_{x\in \bar D}(v,a(x)v)$.
From \cite[Theorem 2.2-ii]{P95}, it follows that
\begin{equation}
\mathcal{P}^\delta(\max_{0\le s\le t}|Y_1(s)-x|\ge \lambda|Y_1(0)=x)\le 2d\exp(-\frac{(\lambda-\delta Bt)^2}{2d\delta At}),
\ \text{for}\ \lambda>\delta Bt.
\end{equation}
Letting $\lambda=\text{dist}(x,\partial D)$ and $t=(\text{dist}(x,\partial D))\delta^{-\frac12}$ in the above inequality and substituting the resulting estimate on
the right hand side of \eqref{bound}, we obtain
$$
u_{\delta, V}(x)\le \exp(-\delta^{-\frac12}\text{dist}(x,\partial D)V_{\text{min}})+
2d\exp(-\delta^{-\frac12}\text{dist}(x,\partial D)\frac{(1-\delta^\frac12B)^2}{2dA}),
$$
for small $\delta>0$, from which the lemma follows.
\end{proof}

The key result for proving Theorem \ref{th} is the following proposition, whose proof is postponed to section 3.
\begin{proposition}\label{key}
Let the assumptions of   Theorem \ref{th} be satisfied for some  $k\ge0$.
Let $j\in\{1,\cdots, m\}$.
If $k$ is even, then
\begin{equation*}
\lim_{\delta\to0}\delta^{-\frac{k+1}2}\int_{A_j}u_{\delta,V}d\mu=\frac1{\sqrt2}\int_{A_j\cap\partial D}\sqrt{(n\cdot an)}
V^{-\frac{k+1}2}\tilde L^{\frac k2}\mu d\sigma.
\end{equation*}
If $k$ is odd, then
\begin{equation*}
\lim_{\delta\to0}\delta^{-\frac{k+1}2}\int_{A_j}u_{\delta,V}d\mu=\frac12\int_{A_j\cap\partial D}V^{-\frac{k+1}2}a\nabla(\tilde L^{\frac{k-1}2}\mu)\cdot nd\sigma.
\end{equation*}
\end{proposition}

\begin{lemma}\label{close}

Let $\mu$ satisfy   the assumptions in  Theorem \ref{th} for some $k\ge0$.
Let $j\in\{1,\cdots, m\}$ be such that $e_0^k(A_j\cap\partial D)>0$.  Then
\begin{equation*}
\lim_{\delta\to0}\mathcal{P}^\delta(Y_1(\tau_{D,1})\in
A_j\cap\partial D|Y_1(0)\in A_j,\tau_{D,1}<J_1)=1.
\end{equation*}
\end{lemma}
\begin{proof}
Define
$$
u_{\delta,V,j}(x)=\mathcal{E}^\delta(1_{\partial D-A_j}(Y_1(\tau_D))\exp(-\int_0^{\tau_{D,1}}V(Y_1(t)))|Y_1(0)=x).
$$
An argument just like that used in the proof of Lemma \ref{intu} shows that
\begin{equation}\label{intuj}
\mathcal{P}^\delta(Y_1(0)\in A_j,\tau_{D,1}<J_1, Y_1(\tau_{D,1})\in \partial D-A_j)=\int_{A_j}u_{\delta,V,j}~d\mu.
\end{equation}
An argument just like that used in the proof of Lemma \ref{expdecay}
shows that
\begin{equation}\label{decayj}
u_{\delta, V,j}(x)\le c\exp(-\frac{\text{dist}(x,\partial D-A_j)}{c\delta^{\frac12}}).
\end{equation}
Let $N>0$ be a positive integer and let $A_j^\frac1N=\{x\in A_j:\text{dist}(x,\partial A_j)<\frac1N\}$.
By Lemma \ref{intu} and \eqref{intuj}, we have
\begin{equation}\label{fraction}
\begin{aligned}
&\mathcal{P}^\delta(Y_1(\tau_{D,1})\in
\partial D-A_j|Y_1(0)\in A_j,\tau_{D,1}<J_1)=\\
&\frac{\int_{A_j^\frac1N}u_{\delta, V,j}d\mu+\int_{A_j-A_j^\frac1N}u_{\delta,V,j}d\mu}
{\int_{A_{j}}u_{\delta, V}d\mu}
\end{aligned}
\end{equation}
Proposition \ref{key}  of course also holds  with $A_j$ replaced by $A_j^\frac1N$.
Using the fact that $u_{\delta,V,j}\le u_{\delta,V}$,
applying Proposition
\ref{key} with $A_j$ and with $A_j^\frac1N$,
and using \eqref{decayj}, we obtain
\begin{equation}\label{final3}
\limsup_{\delta\to0}\mathcal{P}^\delta(Y_1(\tau_{D,1})\in
\partial D-A_j|Y_1(0)\in A_j,\tau_{D,1}<J_1)\le\frac{e_0^k(A_j^\frac1N\cap\partial D)}{e_0^k(A_j\cap\partial D)}.
\end{equation}
Letting $N\to\infty$ completes the proof of the lemma.
\end{proof}

We can now prove \eqref{toshow}, which will complete the proof of Theorem \ref{th}.
 Recall that by assumption $e_0^k(A_1)>0$. Thus, by
 Lemmas \ref{intu} and \ref{close}, it follows
that
\begin{equation}\label{using1}
\mathcal{P}^\delta(Y_1(0)\in A_{1}, Y_1(\tau_D)\in A_{1}\cap\partial D,\tau_{D,1}<J_1)=\big(1+o(1)\big)\int_{A_{1}}u_{\delta,V}d\mu,
\ \text{as}\ \delta\to0.
\end{equation}
By Lemmas \ref{intu} and \ref{close}  and Proposition \ref{key}, it
follows that
\begin{equation}\label{littleo}
\mathcal{P}^\delta(Y_1(0)\in A_j,Y_1(\tau_D)\in  A_1\cap\partial D,\tau_{D,1}<J_1)=o(\delta^{\frac{k+1}2}),\ \text{for}\ j\in\{2,\cdots, m\}.
\end{equation}
Using \eqref{using1}, \eqref{littleo} and  Lemma \ref{expdecay},  we have
\begin{equation}\label{beginning}
\begin{aligned}
&e_\delta(A_1\cap\partial D)=\mathcal{P}^\delta(Y_1(\tau_{D,1})\in A_1\cap\partial D|\tau_{D,1}<J_1)=\\
&\frac{\mathcal{P}^\delta(Y_1(\tau_D)\in  A_1\cap\partial D,\tau_{D,1}<J_1)}{\mathcal{P}^\delta(\tau_{D,1}<J_1)}=\\
&\frac{\sum_{j=1}^{m+1}\mathcal{P}^\delta(Y_1(0)\in A_j,Y_1(\tau_D)\in  A_1\cap\partial D,\tau_{D,1}<J_1)}
{\sum_{j=1}^{m+1}\mathcal{P}^\delta(Y_1(0)\in A_j,\tau_{D,1}<J_1)}=\\
&\frac{\big(1+o(1)\big)\int_{A_1}u_{\delta,V}~d\mu+o(\delta^\frac{k+1}2)}
{\sum_{j=1}^m \int_{A_j}u_{\delta,V}~d\mu +o(\delta^\frac{k+1}2)},\ \text{as}\ \delta\to0.
\end{aligned}
\end{equation}
Now from \eqref{beginning} and Proposition \ref{key}, we have
\begin{equation}
\lim_{\delta\to0}e_\delta(A_1\cap\partial D)= e_0^k(A_1\cap \partial D).
\end{equation}

\hfill $\square$

\section{Proof of Proposition \ref{key}}
For the proof of the proposition in the case of even $k$, we will need the following lemma.
\begin{lemma}\label{boundary}
Let $n$ denote the unit inward normal to $D$ at $\partial D$. One has
\begin{equation*}
\lim_{\delta\to0}\delta^\frac12(n\cdot a\nabla u_{\delta,V})(x)=-\sqrt{2V(x)(n\cdot an)(x)},\ \text{uniformly over}\ x\in\partial D.
\end{equation*}
\end{lemma}
\begin{proof}
The proof of the result in the case that $L=\frac12\Delta$ was given in \cite{P09}.
In the proof, it was shown that everything could be reduced to local considerations.
In particular, it was enough to prove that the above equation holds pointwise under the assumption that the boundary had constant curvature.
We can thus make the same assumptions here, and we can also assume that $L$ and $V$ have  constant coefficients.
More specifically, note that  $L$ is given in non-divergence by
\begin{equation}
\begin{aligned}
&L=\frac12\sum_{i,j=1}^da_{i,j}(x)\frac{\partial^2}{\partial x_i \partial x_j}+\frac12\sum_{i,j=1}^d \frac{\partial a_{i,j}}{\partial x_i}(x)
\frac{\partial }{\partial x_j}+\sum_{j=1}^db_j(x)\frac\partial{\partial x_j}\equiv\\
&\frac12\sum_{i,j=1}^da_{i,j}(x)\frac{\partial^2}{\partial x_i \partial x_j}+\sum_{i=1}^dB_j(x)\frac\partial{\partial x_j}.
\end{aligned}
\end{equation}
Thus for the proof we may assume that
$L=\frac12\sum_{i,j=1}^da_{i,j}\frac{\partial^2}{\partial x_i \partial x_j}+\sum_{i=1}^dB_j\frac\partial{\partial x_j}$,
where $a_{i,j}$ and $B_j$ are constant.
Similar to what was done in \cite{P09}, for zero curvature, we  take $D=\{x\in R^d:0<x_1<1\}$ and consider
$(n\cdot a\nabla u_{\delta,V})(0)$; for curvature $R>0$,
we assume that $D=A_{\frac R2,R}
\equiv\{x\in R^d:\frac R2<|x|<R\}$ and consider
$(n\cdot a\nabla u_{\delta,V})(x)$, for some $x$ with $|x|=R$;
and for negative curvature $-R<0$, we assume that
$D=A_{R,2R}\equiv\{x\in R^d:R<|x|<2R\}$
and consider $(n\cdot a\nabla u_{\delta,V})(x)$, for some $x$ with $|x|=R$.
We will consider the cases of zero curvature and positive curvature; the case of negative curvature
being handled similarly to the case of positive curvature.

We begin with the case of zero curvature.
Let $\{e_j\}_{j=1}^d$ denote the standard basis vectors.
Let $H_1$ denote the hyperplane $x_1=0$.
The interior unit normal to $D$ on $\partial D\cap H_1$ is constant and equal to $e_1$; that is,

$n=e_1$.
Let
 $y$ denote the projection of $n\cdot a$ onto $H_1$.
Then $y+(n\cdot an)e_1=n\cdot a$. Since $u_{\delta,V}=1$ on $H_1$, we have
\begin{equation}\label{boungrad}
\begin{aligned}
&(n\cdot a\nabla u_{\delta,V})(0)=\lim_{t\to0+}\frac{u_{\delta,V}(tn\cdot a)-u_{\delta,V}(0)}t=\lim_{t\to0+}\frac
{u_{\delta,V}(tn\cdot a)-u _{\delta,V}(ty)}t=\\
&\lim_{t\to0^+}\frac{u_{\delta,V}(ty+t(n\cdot an)e_1)-u_{\delta,V}(ty)}t=n\cdot an\lim_{t\to0^+}\frac{\partial u_{\delta,V}}{\partial x_1}(ty+s_t(n\cdot an)e_1)=\\
&(n\cdot an)(n\cdot \nabla u_{\delta,V})(0),
\end{aligned}
\end{equation}
where $0<s_t<t$. Since $u_{\delta,V}$ depends only on $x_1$ and since $n=e_1$, we can reduce the calculation
of $(n\nabla u_{\delta,V})(0)$   to a one-dimensional problem.
So we write $ u_{\delta,V}=u_{\delta,V}(x)$ with $0<x<1$. Now  $u$ solves
the constant coefficient equation $\frac12\delta a_{1,1}u_{\delta,V}''+\delta B_1u_{\delta,V}'-Vu=0$
with the boundary condition $u_{\delta,V}(0)=u_{\delta,V}(1)=1$.
The quantity $(n\cdot \nabla u_{\delta,V})(0)$ above is now given by
$u'_{\delta,V}(0)$.
One can solve this explicitly and check that
$\lim_{\delta\to0}\delta^\frac12u_{\delta,V}'(0)=-\sqrt{\frac{2V}{a_{1,1}}}$.
Substituting this in \eqref{boungrad} and noting that $a_{1,1}=n\cdot an$,
we obtain
$\lim_{\delta\to0}(n\cdot a\nabla u_{\delta,V})(0)=-\sqrt{2(n\cdot an)V}$.

Now we turn to the case that the curvature is $R>0$. We let $D=A_{\frac R2,R}$ and consider the boundary point
$Re_1$. We need to evaluate $\lim_{\delta\to0}(n\cdot a\nabla u_{\delta,V})(Re_1)$.
 We first reduce the calculation to the calculation of the normal derivative, similar to \eqref{boungrad}.
 Note that the inward unit normal $n=n(Re_1)$ at $Re_1$ satisfies $n=-e_1$.
For small $t>0$, let $z_t$ denote the point on $|x|=R$ which is closest to $Re_1+t n\cdot a$.
Define the vector $w_t$ by $z_t+w_t=Re_1+tn\cdot a$.
(Note that $Re_1$, $z_t$  and $w_t$ take on the roles played by   0, $ty$ and $t(n\cdot an)e_1$ respectively
in the case of zero curvature.)
Of course
$\lim_{t\to0^+}|z_t-Re_1|=0$. Since the curvature is positive, we have
$|w_t|<(n\cdot an)t$; however $\lim_{t\to0+}\frac{|w_t|}t=n\cdot an$.
Note also that the direction  $\frac{w_t}{|w_t|}$ of $w_t$ approaches the direction of $n$ as $t\to0^+$.
Thus, since $u_{\delta,V}=1$ on $|x|=R$,  we have
\begin{equation}\label{boungradcurv}
\begin{aligned}
&(n\cdot a\nabla u_{\delta, V})(Re_1)=\lim_{t\to0^+}\frac{u_{\delta,V}(Re_1+tn\cdot a)-u _{\delta,V}(Re_1)}t=\\
&\lim_{t\to0+}\frac{u _{\delta,V}(z_t+w_t)-u _{\delta,V}(z_t)}t
=\lim_{t\to0^+}\frac{|w_t|}t\frac{u _{\delta,V}(z_t+w_t)-u _{\delta,V}(z_t)}{|w_t|}=\\
&(n\cdot an)(n\cdot \nabla u_{\delta, V})(Re_1).
\end{aligned}
\end{equation}
We now consider $(n\cdot\nabla u_{\delta, V})(Re_1)$.
Let $(r,\theta)$ with $\theta\in S^{d-1}$ denote polar coordinates.
We rewrite the constant coefficient operator $L$ in polar form. Of course now the operator will no longer
have constant coefficients; however by the localization mentioned above, we may consider
instead the constant coefficient  operator obtained by
evaluating the  coefficients at $Re_1$.
Call the resulting operator $L$.
We have $L=\frac12(n\cdot an)\frac{d^2}{dr^2}+B\frac d{dr}+\text{terms involving differentiation with respect to}\ \theta
\ \text{and maybe also }\ r$,
where $B$ is a certain constant whose form is irrelevant for our purposes.
Now $u_{\delta, V}$ solves
 $\delta Lu_{\delta,V}-Vu_{\delta,V}=0$ for $\frac R2<r<R$, and
$u_{\delta,V}=1$ at $r=\frac R2$ and $r=R$. It follows that $u_{\delta, V}$ is a function of $r$ alone.
Thus $u_{\delta, V}$ satisfies the one-dimensional equation
$\frac12\delta (n\cdot an)u''_{\delta,V}+\delta Bu'_{\delta,V}-Vu_{\delta,V}=0$ for $\frac R2<r<R$
and $u(\frac R2)=u(R)=1$, and
$(n\cdot\nabla u_{\delta, V})(Re_1)$
becomes $-u'_{\delta, V}(R)$.
 We have thus reduced the problem
to the previous case of zero curvature, and conclude that
$\lim_{\delta\to0}(n\cdot a\nabla u_{\delta,V})(Re_1)=-\sqrt{2(n\cdot an)V}$.

\end{proof}

\it\noindent Proof of Proposition 1.\rm\
Let $\mu_0(\cdot)$ be an arbitrary probability measure on $\bar D$
which has a density $\mu_0(x)$ which
satisfies the same smoothness assumptions in $\bar D$ that the density $\mu$ satisfies in $D^\epsilon$, and which satisfies
the same vanishing conditions on $\partial D$ that the density $\mu$ satisfies there.
An easy argument then shows that to prove the proposition, it suffices to prove it with $A_j$ replaced by $\bar D$, $d\mu$ replaced
by $d\mu_0$ and $\mu(x)$ replaced by $\mu_0(x)$.
We will first prove the proposition for the case $k=1$, which is easier than the case $k=0$.  We then show how to go from the case
$k=1$ to the case $k=3$, from which it will be clear how to proceed for odd $k$.
After that we will prove the proposition for $k=0$ and then we show how to go
from  the case $k=0$ to the case  $k=2$, from which it will be clear how to proceed for even $k$.

In light of the above paragraph, we consider $\int_Du_{\delta, V}\mu_0dx$.
Since $k=1$, $\mu_0$ vanishes on $\partial D$, but $\nabla \mu_0$ does not vanish identically
on $\partial D$.
Using \eqref{u} and the fact that $\mu_0$ vanishes on $\partial D$,
and recalling that $n$ denotes the inward unit normal,
integration by parts gives
\begin{equation}\label{k=1}
\begin{aligned}
&\delta^{-1}\int_Du_{\delta, V}\mu_0dx=\int_DLu_{\delta, V}
\frac{\mu_0}Vdx=\\
&\int_Du_{\delta, V}
\tilde L\frac{\mu_0}Vdx+
\frac12\int_{\partial D}a\nabla(\frac{\mu_0}V)\cdot nd\sigma.
\end{aligned}
\end{equation}
(Note that by assumption, $\mu_0$ and $V$ are $C^2$-functions so there is no problem with the integration by parts.)
By Lemma \ref{expdecay}, $u_{\delta, V}$ converges to 0 boundedly pointwise on $D$.
Also, since $\mu_0$ vanishes on $\partial D$, we have $\nabla(\frac{\mu_0}V)\cdot n=\frac1V\nabla \mu_0\cdot n$
on $\partial D$. Thus, letting $\delta\to0$ in \eqref{k=1}, we obtain
\begin{equation*}
\lim_{\delta\to0}\delta^{-1}\int_Du_{\delta, V}\mu_0dx=
\frac12\int_{\partial D}V^{-1}a\nabla\mu_0\cdot nd\sigma.
\end{equation*}
We now turn to the case $k=3$.
In the case $k=3$, $\mu_0$ and all its derivatives up to order 2 vanish on $\partial D$;
in particular, the last term on the right hand side of \eqref{k=1} is 0.
Thus, using \eqref{u} again, integrating by parts and using the fact that the second order derivatives of $\mu_0$ vanish
on $\partial D$, we have from \eqref{k=1},
\begin{equation}\label{k=3}
\begin{aligned}
&\delta^{-2}\int_Du_{\delta,V}\mu_0dx=
\delta^{-1}\int_Du_{\delta, V}
\tilde L\frac{\mu_0}Vdx=
\int_D(Lu_{\delta,V})\frac1V\tilde L\frac{\mu_0}Vdx=\\
&\int_Du_{\delta,V}\tilde L\frac1V\tilde L\frac{\mu_0}Vdx+
\int_{\partial D}
a\nabla(\frac1V\tilde L\frac{\mu_0}V)\cdot nd\sigma.
\end{aligned}
\end{equation}
(Note that by assumption, $\mu_0$ and  $V$ are $C^4$-functions  and $a_{i,j}$ and $b_i$ are $C^3$-functions, so
there is no problem with the integration by parts.)
Using Lemma \ref{expdecay} again and the fact that $\mu_0$ and all its derivatives up to order 2 vanish on $\partial D$,
we obtain
$$
\lim_{\delta\to0}\delta^{-2}\int_Du_{\delta,V}\mu_0dx=\int_{\partial D}
V^{-2}a\nabla(\tilde L\mu_0)\cdot nd\sigma.
$$
The same technique is used repeatedly to handle larger values of odd $k$, the smoothness requirements in the statement of
Theorem \ref{th} being the smoothness required to implement the integration by parts.

Now we turn to the case $k=0$.
Let $w$ solve the equation
\begin{equation}\label{aux}
\begin{aligned}
&\tilde L\frac w{ V}=0\ \text{in}\ D;\\
&w=\mu_0\ \text{on}\ \partial D.
\end{aligned}
\end{equation}
Note that by the smoothness assumptions on $a,b,V,\mu_0$, it follows that $w$ is the solution to an
elliptic  equation with $C^\alpha(\bar D)$-coefficients and continuous boundary data.
Thus, $w\in C^{2,\alpha}(D)\cap C(\bar D)$.

We will show below that
\begin{equation}\label{diff}
\lim_{\delta\to0}\delta^{-\frac12}\int_Du_{\delta, V}(\mu_0-w)dx=0.
\end{equation}
Thus, it is enough to show  that
\begin{equation}\label{show}
\lim_{\delta\to0}\delta^{-\frac12}\int_{D}u_{\delta,V}wdx=\frac1{\sqrt{2}}\int_{\partial D}\sqrt{(n\cdot an)}\frac{\mu_0}{\sqrt{V}}d\sigma.
\end{equation}
Using \eqref{u} and \eqref{aux}, and integrating by parts, we have
\begin{equation}\label{k=0}
\begin{aligned}
&\delta^{-\frac12}\int_Du_{\delta, V}wdx=\delta^{\frac12}\int_D
(L u_{\delta, V})\frac wVdx=
-\frac{\delta^{\frac12}}2\int_{\partial D}\frac{\mu_0}V
a\nabla u_{\delta, V}\cdot n d\sigma,
\end{aligned}
\end{equation}
where we have used the fact that
\begin{equation*}\label{keycalc}
\frac12\int_{\partial D}a\nabla( \frac w V)\cdot nd\sigma-\int_{\partial D}\frac wVb\cdot nd\sigma=
\int_D \tilde L\frac wVdx=0
\end{equation*}
by \eqref{aux}.
(Actually, since $w$ is not necessarily $C^2$ up to the boundary, in the above integrals one should replace $D$ by
    $D-\bar D^\epsilon$  and $\partial D$ by $\partial (D-\bar D^\epsilon)$ and then let  $\epsilon\to0$.)
 Letting $\delta\to0$ in \eqref{k=0}, and using Lemma \ref{boundary}, we obtain
\eqref{show}.

It remains to prove \eqref{diff}.
By Lemma \ref{expdecay}, we have
\begin{equation}\label{diffbig}
|\delta^{-\frac12}\int_{D-D^\epsilon}u_{\delta, V}(\mu_0-w)dx|\le\sup_{x\in D}
(\mu_0(x)+w(x))|D|
\delta^{-\frac12}c\exp(-\frac{\epsilon}{c\delta^{\frac12}}).
\end{equation}
We also have
\begin{equation}\label{difflittle}
|\delta^{-\frac12}\int_{D^\epsilon}u_{\delta,V}(\mu_0-w)dx|\le\sup_{x\in D^\epsilon}
|\mu_0(x)-w(x)|
(\delta^{-\frac12}\int_Du_{\delta,V}dx).
\end{equation}
Now  \eqref{show} holds for every $\mu_0$ in a wide class; in particular,  it holds for $\mu_0$ which are uniformly positive on
$\bar D$. In such a case, it follows by the maximum principal that $w$ is uniformly positive on $\bar D$.
(The principal eigenvalue for $\tilde L$ coincides with that of $L$, and is consequently negative. Thus the generalized maximum
principal holds: $\tilde L v=0$ in $D$ and $v>0$ on $\partial D$ guarantees that $v>0$ on $\bar D$. Apply this with $v=\frac wV$.)
By considering \eqref{show} with such a uniformly positive $w$, it follows that  $\delta^{-\frac12}\int_Du_{\delta,V}dx$ is bounded as $\delta\to0$.
Using this, the proof of \eqref{diff} now follows from
\eqref{diffbig}, \eqref{difflittle} and the fact that $\lim_{\epsilon\to0}\sup_{x\in D^\epsilon}
|\mu_0(x)-w(x)|=0$.

We now turn to the case $k=2$. Since $\mu_0$ and all its derivatives up to order one vanish on $\partial D$, we can write
\eqref{k=1} as
\begin{equation}\label{reduction}
\delta^{-\frac32}\int_Du_{\delta, V}\mu_0dx=
\delta^{-\frac12}\int_Du_{\delta, V}
\tilde L\frac{\mu_0}Vdx.
\end{equation}
As with the case $k=0$, we define an auxiliary function $w$. Let $w$ solve the equation
\begin{equation}
\begin{aligned}\label{auxagain}
&\tilde L\frac wV=0\ \text{in}\ D;\\
&w=\tilde L\frac{\mu_0}V\ \text{on}\ \partial D.
\end{aligned}
\end{equation}
(By assumption, $\mu_0$ and its first order partial derivatives vanish on $\partial D$, but not all of its second order partial
derivatives vanish on $\partial D$. It then follows from the maximum principal that $\tilde L\frac {\mu_0}V\gneq0$ on $\partial D$.)
The same argument used to show \eqref{diff} shows that
\begin{equation}\label{diffagain}
\lim_{\delta\to0}\delta^{-\frac12}\int_Du_{\delta, V}(\tilde L\frac{\mu_0}V-w)\mu_0dx=0.
\end{equation}
In light of \eqref{reduction} and \eqref{diffagain},  it is enough to prove that
\begin{equation}\label{show2}
\lim_{\delta\to0}\delta^{-\frac12}\int_Du_{\delta,V}wdx=
\frac1{\sqrt2}\int_{\partial D}\sqrt{(n\cdot an)}
V^{-\frac32}\tilde L\mu_0 d\sigma.
\end{equation}
Using \eqref{u}, integrating by parts and using \eqref{auxagain}, we have
\begin{equation}\label{k=2}
\delta^{-\frac12}\int_Du_{\delta,V}wdx=\delta^{\frac12}\int_D (Lu_{\delta,V})\frac wVdx=
-\frac{\delta^\frac12}2\int_{\partial D}\frac1V(\tilde L\frac {\mu_0}V)a\nabla u_{\delta,V}\cdot nd\sigma,
\end{equation}
where we have  used the fact that
\begin{equation*}
\frac12\int_{\partial D}a\nabla( \frac w V)\cdot nd\sigma-\int_{\partial D}\frac wVb\cdot nd\sigma=
\int_D \tilde L\frac wVdx=0
\end{equation*}
by \eqref{auxagain}.
 Since $\mu_0$ and all its first order partial derivatives vanish on $\partial D$, we have $\tilde L\frac{\mu_0}V=\frac1V\tilde L\mu_0$ on $\partial D$.
Using this and  Lemma \ref{boundary}, and letting $\delta\to0$ in \eqref{k=2},  we obtain
\eqref{show2}.
The same technique is used repeatedly to handle larger values of even $k$, the smoothness requirements in the statement of
Theorem \ref{th} being the smoothness required to implement the integration by parts.
\hfill $\square$

\end{document}